\documentclass[12pt]{amsart}
\usepackage{amssymb,amsfonts,latexsym,amscd}

\usepackage[all]{xy}
\usepackage{verbatim}

\newtheorem{theorem}{Theorem}[section]
\newtheorem{lemma}[theorem]{Lemma}
\newtheorem{proposition}[theorem]{Proposition}
\newtheorem{corollary}[theorem]{Corollary}
\theoremstyle{definition}
\newtheorem{definition}[theorem]{Definition}

\newtheorem{example}[theorem]{Example}
\newtheorem{question}[theorem]{Question}

\newtheorem{remark}[theorem]{Remark}

\newcommand{\End}{\text{End}}
\newcommand{\Fun}{\text{Fun}}

\newcommand{\Pro}{\text{Pro}}

\newcommand{\Ext}{\text{\rm Ext}}

\newcommand{\Hom}{\text{Hom}}
\newcommand{\Ad}{\text{Ad}}

\newcommand{\Comod}{\text{Comod}}

\newcommand{\Ind}{\text{Ind}}

\newcommand{\Rep}{\text{Rep}}

\newcommand{\Spec}{\text{Spec}}
\newcommand{\Corep}{\text{Corep}}
\newcommand{\Vect}{\text{Vec}}

\newcommand{\ot}{\otimes}

\newcommand{\ben}{\begin{enumerate}}
\newcommand{\een}{\end{enumerate}}

\newcommand{\C}{{\mathcal C}}

\newcommand{\Mod}{\mbox{Mod}}

\begin{document}

\title[Module
categories over affine group schemes] {Module categories over affine
group schemes}

\author{Shlomo Gelaki}
\address{Department of Mathematics, Technion-Israel Institute of
Technology, Haifa 32000, Israel} \email{gelaki@math.technion.ac.il}

\date{\today}


\begin{abstract}
Let $k$ be an algebraically closed field of characteristic $p\ge 0$. Let $G$ be an affine group
scheme over $k$. We classify the indecomposable exact module categories over the rigid
tensor category $\text{Coh}_f(G)$ of coherent sheaves of finite dimensional
$k-$vector spaces on $G$, in terms of $(H,\psi)-$equivariant coherent sheaves on $G$. We deduce from it the
classification of indecomposable {\em geometrical} module categories over $\Rep(G)$. When $G$ is finite, this yields the
classification of {\em all} indecomposable exact module categories over the finite tensor category $\Rep(G)$. In particular, we obtain a
classification of twists for the group algebra $k[G]$ of a finite
group scheme $G$. Applying this to $u(\mathfrak {g})$, where
$\mathfrak {g}$ is a finite dimensional $p-$Lie algebra over $k$
with positive characteristic, produces (new) finite dimensional
noncommutative and noncocommutative triangular Hopf algebras in
positive characteristic. We also introduce and study group scheme theoretical categories, and study isocategorical finite group
schemes.
\end{abstract}

\maketitle

\section{introduction}

Let $k$ be an algebraically closed field of characteristic $p\ge 0$.
Let $G$ be a finite group. Consider the fusion category $\Vect(G)$ of finite dimensional $G-$graded vector spaces over $k$, and the finite tensor category $\Rep(G)$
of finite dimensional representations of $G$ over $k$. Etingof and Ostrik classified the indecomposable exact
module categories over $\Rep(G)$ \cite{EO}, generalizing the classification of
Ostrik in zero characteristic \cite{O}. Alternatively, one could use the duality between $\Vect(G)$ and $\Rep(G)$ (provided by the usual fiber functor on $\Rep(G)$) and the classification of the indecomposable exact
module categories over $\Vect(G)$ to obtain the same result. In particular, the
classification of the semisimple module categories of rank $1$ provides
the
classification of twists for the group algebra $k[G]$, reproducing the classification
given by Movshev in zero characteristic \cite{Mov}. The
classification of twists for finite groups, together with Deligne's
theorem \cite{DE}, enabled Etingof and the author to classify
triangular semisimple and cosemisimple Hopf algebras over $k$ \cite{EG} (see
also \cite{G}).

The goal of this paper is to extend the classification of Etingof
and Ostrik mentioned above to \emph{finite group schemes} $G$ over
$k$, and in particular thus obtain (new) finite dimensional
noncommutative and noncocommutative triangular Hopf algebras in
positive characteristic by twisting $k[G]$. However, in absence of
Deligne's theorem in positive characteristic, the classification of
finite dimensional triangular Hopf algebras in positive
characteristic remains out of reach.

Let $G$ be a finite group scheme over $k$. The idea is to first classify the indecomposable exact module
categories over $\Rep(k[G]^*)$, where $k[G]^*$ is the dual Hopf
algebra of the group algebra $k[G]$ of $G$, and then use the fact
that they are in bijection with the indecomposable exact module
categories over $\Rep(G)$ \cite{EO} to get the classification of the
latter ones. The reason we approach it in this way is that $k[G]^*$ is just the Hopf
algebra $\mathcal{O}(G)$ representing the group scheme $G$, so
$\Rep(k[G]^*)$ is tensor equivalent to the tensor category
$\text{Coh}_f(G)=\text{Coh}(G)$ of coherent sheaves of \emph{finite dimensional}
$k-$vector spaces on $G$ with the tensor product of convolution of
sheaves, which allows us to use geometrical tools and arguments. For
example, when $G$ is an abstract finite group,
$\text{Coh}(G)=\Vect(G)$. 

In fact, in Theorem \ref{grsch} we
classify the indecomposable exact module categories over
$\text{Coh}_f(G)$, where $G$ is \emph{any} affine group scheme over $k$ (i.e., $G$ is not necessarily finite). The classification is given in terms of $(H,\psi)-$equivariant coherent sheaves on $G$ (see Definition \ref{defequiv}). Since $\text{Coh}_f(G)$ is no longer finite when $G$ is not, the proof requires working with Ind and Pro objects, which makes it technically more involved. Furthermore, when
$G$ is not finite, not all indecomposable exact module categories
over $\Rep(G)$ are obtained from those over $\text{Coh}_f(G)$ (see
Theorem \ref{modg} and Remark \ref{notall}); we refer to  those which are as {\em geometrical}. So the classification of module categories (even fiber functors) over $\Rep(G)$ for infinite affine group schemes $G$ remains unknown (even when $G$ is a linear algebraic group over $\mathbb{C}$). 

In Section \ref{gth} we introduce the class of {\em group scheme theoretical categories}, which extends both $\text{Coh}_f(G)$ and $\Rep(G)$, and generalize to them the results from Sections \ref{cohg} and \ref{repg} mentioned above.

As a consequence of our results, combined with \cite{AEGN, EO}, we obtain in Corollary \ref{twistsfinite} that gauge equivalence classes of twists
for the group algebra $k[G]$ of a \emph{finite} group scheme $G$ over $k$ are
parameterized by conjugacy classes of pairs $(H,J)$, where $H$ is a closed group
subscheme of $G$ and $J$ is a \emph{nondegenerate} twist for $k[H]$ (
just as in the case of abstract finite groups). Furthermore, in Proposition \ref{minond} we show that a twist for $G$ is nondegenerate if and only if it is \emph{minimal} (again, as for abstract finite groups), by showing directly, that is, without using Deligne's theorem, that a quotient of a Tannakian category is also Tannakian (Proposition \ref{quot}). We use this in
Sections \ref{excom}, \ref{explie} to give some examples of twists for $k[A]$ and
$u(\mathfrak {g})$, where $A$ is a finite commutative group scheme
over $k$ with positive characteristic and $\mathfrak {g}$ is a
finite dimensional $p-$Lie algebra over $k$ of positive
characteristic $p$. In particular, applying this to $u(\mathfrak {g})$ yields (new) finite dimensional
noncommutative and noncocommutative triangular Hopf algebras in positive characteristic.

We conclude the paper with Section \ref{isoc} in which we extend \cite{EG1} by giving the
construction of all finite group schemes which are isocategorical to
a fixed finite group scheme. In particular, it follows that two isocategorical finite group schemes are necessarily isomorphic as schemes (but not as groups \cite{EG1}, \cite{Da1, Da2}).


{\bf Acknowledgments.} The author is grateful to Pavel Etingof for stimulating and helpful discussions.

The research was partially supported by The Israel Science
Foundation (grants No. 317/09 and 561/12).

\section{Preliminaries}

Throughout the paper we fix an algebraically closed field $k$ of
characteristic $p\ge 0$.

\subsection{Affine group schemes}\label{gsch} Let $G$ be an affine group scheme over $k$, with unit morphism $e:\Spec(k)\to G$, inversion morphism
$i:G\to G$, and multiplication morphism $m:G\times G\to G$,
satisfying the usual group axioms. Let $\mathcal{O}(G)$ be the coordinate algebra of $G$, i.e., $\mathcal{O}(G)$ is a commutative Hopf algebra, and
$G=\Spec_k(\mathcal{O}(G))$. In other words, we are given a collection of group structures on the sets $G(R):=\Hom_{Alg}(\mathcal{O}(G),R)$ of $R-$valued points of $G$, where $R$ is a commutative algebra over $k$, which is functorial in $R$. (See, e.g., \cite{JJ} for the general theory of group schemes.)

A {\em closed} group subscheme of $G$ is, by definition, an affine group scheme $H$ whose coordinate algebra $\mathcal{O}(H)$ is a quotient of $\mathcal{O}(G)$ by a radical Hopf ideal $\mathcal{I}(H)$: $\mathcal{O}(H)=\mathcal{O}(G)/\mathcal{I}(H)$. The ideal $\mathcal{I}(H)$ is referred to as the {\em defining ideal} of $H$ in $\mathcal{O}(G)$. For example, the
{\em connected component of the identity} in $G$ is a closed group subscheme of $G$, denoted by $G^0$. Let
$\pi_0(G):=G/G^0$. Then $\mathcal{O}(\pi_0(G))$ is the unique
maximal finite dimensional semisimple Hopf subalgebra of
$\mathcal{O}(G)$. 

We let $\Rep(G)$ denote the category of finite dimensional rational representations of $G$ over $k$ (equivalently, $\Rep(G)$ is the category of finite dimensional comodules over $\mathcal{O}(G)$); it is a
\emph{symmetric rigid tensor category}. (See, e.g., \cite{e} for the definition of a tensor category
and its general theory.)

\subsubsection{Finite group schemes.} An affine group scheme $G$ is called \emph{finite} if
$\mathcal{O}(G)$ is finite dimensional. In this case,
$\mathcal{O}(G)^*$ is a finite dimensional cocommutative Hopf
algebra, which is called the \emph{group algebra} of $G$, and
denoted by $k[G]$. In particular, $\Rep(G)$ is a \emph{finite}
symmetric tensor category and $\Rep(G)=\Rep(k[G])$ as symmetric
tensor categories. A finite group scheme $G$ is called
\emph{constant} if its representing Hopf algebra $\mathcal{O}(G)$ is
the Hopf algebra of functions on some finite
abstract group with values in $k$, and is called \emph{etale} if $\mathcal{O}(G)$ is
semisimple. Since $k$ is algebraically closed, it is known that $G$
is etale if and only if it is a constant group scheme \cite{W}. A
finite group scheme $G$ is called \emph{infinitesimal} if
$\mathcal{O}(G)$ is a local algebra.

\begin{theorem} (See \cite[6.8, p.52]{W})
Let $G$ be a finite group scheme. Then $\pi_0(G)$ is etale, $G^0$ is
infinitesimal, and $G$ is a semidirect product $G = G^0\rtimes
\pi_0(G)$. If the characteristic of $k$ is $0$ then $G=\pi_0(G)$ is etale.
\end{theorem}

Let $G$ be a finite \emph{commutative} group scheme over $k$, i.e.,
$\mathcal{O}(G)$ is a finite dimensional commutative and
cocommutative Hopf algebra. In this case, $k[G]$ is also a finite dimensional
commutative and
cocommutative Hopf algebra, so it represents a finite commutative
group scheme $G^D$ over $k$, which is called the \emph{Cartier dual}
of $G$. For example, the Cartier dual of the group scheme $\alpha_p$
(= the Frobenius kernel of the additive group $\mathbb{G}_a$) is
$\alpha_p$, while the Cartier dual of $\mu_p$ (= the Frobenius
kernel of the multiplicative group $\mathbb{G}_m$) is the constant
group scheme $\mathbb{Z}/p\mathbb{Z}$.

\begin{theorem} (See \cite[6.8, p.52]{W})
Let $G$ be a finite commutative group scheme over $k$. Then
$G=G_{ee}\times G_{ec}\times G_{ce}\times G_{cc}$ decomposes
canonically as a direct product of four finite commutative group
schemes over $k$ of the following types: $G_{ee}$ is etale with
etale dual (i.e., an abstract abelian group $A$ such that $p\nmid |A|$),
$G_{ec}$ is etale with connected dual (e.g.,
$\mathbb{Z}/p\mathbb{Z}$), $G_{ce}$ is connected with etale dual
(e.g., $\mu_p$), and $G_{cc}$ is connected with connected dual
(e.g., $\alpha_p\cong \alpha_p^D$).
\end{theorem}

Recall that a finite commutative group scheme $G$ is called
\emph{diagonalizable} if $\mathcal{O}(G)$ is the group algebra
$k[A]$ of a finite abelian group $A$. For example, $\mu_n$ is
diagonalizable since $\mathcal{O}(\mu_n)=k[\mathbb{Z}/n\mathbb{Z}]$.
In fact, any diagonalizable finite group scheme $G$ is a direct
product of various $\mu_n$. Clearly, the group algebra of a finite
diagonalizable group scheme is semisimple.

\begin{theorem}(Nagata, see \cite[p.223]{A})\label{ss}
Let $G$ be a finite group scheme over $k$. The group algebra $k[G]$ is
semisimple if and only if $G^0$ is diagonalizable and $p$ does not
divide the order of $G(k)$. In particular, for an infinitesimal
group scheme $G$, if $k[G]$ is semisimple then $G$ is
diagonalizable.
\end{theorem}

\subsection{$p-$Lie algebras}\label{plie}
Assume that the ground field $k$ has characteristic $p>0$. Let
$\mathfrak{g}$ be a finite dimensional $p-$Lie algebra over $k$ and
let $u(\mathfrak{g})$ be its $p-$restricted universal enveloping
algebra (see, e.g., \cite{J}, \cite{SF}). Then $u(\mathfrak{g})$ is
a cocommutative Hopf algebra of dimension $p^{\dim(\mathfrak{g})}$
and its dual (commutative) Hopf algebra $u(\mathfrak{g})^*$ is a
local algebra satisfying $x^p=0$ for any $x$ in the augmentation
ideal of $u(\mathfrak{g})^*$. Recall that there is an equivalence of
categories between the category of infinitesimal group schemes
$G$ over $k$ of \emph{height $1$} and the category of finite dimensional
$p-$Lie algebras $\mathfrak{g}$ over $k$, given by
$G\mapsto\mathfrak{g}$, where $k[G]=u(\mathfrak{g})$.

An $n-$dimensional \emph{torus} is an $n-$dimensional abelian
$p-$Lie algebra $\mathfrak{t}$ over $k$ with a basis consisting of
\emph{toral} elements $h_i$ (i.e, $h_i^{p}=h_i$). By a theorem of
Hochschild (see Theorem \ref{ss}), tori are precisely those
finite dimensional $p-$Lie algebras whose representation categories
are semisimple (see \cite{SF}). In other words, $u(\mathfrak{t})$ is
a semisimple commutative (and cocommutative) Hopf algebra, and
$\Rep(\mathfrak{t})=\Rep(u(\mathfrak{t}))$ is a fusion category.
Moreover, it is known that $u(\mathfrak{t})$ is isomorphic to the
Hopf algebra $\Fun((\mathbb{Z}/p\mathbb{Z})^n,k)$ of functions on the
elementary abelian $p-$group of rank $n$, so
$\Rep(u(\mathfrak{t}))=\Vect((\mathbb{Z}/p\mathbb{Z})^n)$ is the fusion
category of finite dimensional $(\mathbb{Z}/p\mathbb{Z})^n-$graded vector
spaces.

\subsection{Module categories over tensor categories.}
Let $\mathcal{C}$ be a rigid tensor category over $k$, i.e., a
$k-$linear locally finite abelian category,
equipped with an associative tensor product, a unit object and a rigid structure (see, e.g., \cite{e}). In particular, every object in $\C$ has {\em finite length}.

Let $\Ind(\mathcal{C})$ and $\Pro(\mathcal{C})$ denote the categories of
$\Ind-$objects and $\Pro-$objects of $\mathcal{C}$, respectively (see, e.g., \cite{KS}). The rigid structure on $\mathcal{C}$ induces two duality functors $\Pro(\mathcal{C})\to\Ind(\mathcal{C})$ (``continuous dual") and $\Ind(\mathcal{C})\to \Pro(\mathcal{C})$ (``linear dual"), which we shall both denote by $X\mapsto X^*$; they are antiequivalence inverses of each other.

It is well known that the tensor
structure on $\mathcal{C}$ extends to a tensor structure on
$\Ind(\mathcal{C})$ and $\Pro(\mathcal{C})$ (however, $\Ind(\mathcal{C})$ and $\Pro(\mathcal{C})$ are not rigid). It
is also known that $\Ind(\mathcal{C})$ has enough injectives. More
generally, recall that a (left) \emph{module category} over $\mathcal{C}$ is a locally finite abelian category $\mathcal{M}$ over $k$ equipped with a (left) action $\ot ^{\mathcal{M}}:\mathcal{C}\boxtimes \mathcal{M}\to \mathcal{M}$, such that the
bifunctor $\ot ^{\mathcal{M}}$ is bilinear on morphisms and biexact.
Similarly, the $\mathcal{C}-$module structure on $\mathcal{M}$ extends to a module structure on $\Ind(\mathcal{M})$ over $\Ind(\mathcal{C})$.

One can define a \emph{dual internal Hom} in a $\mathcal{C}-$module
category $\mathcal{M}$ as follows: for $M_1,M_2\in \mathcal{M}$, let
$\overline{\Hom}(M_1,M_2)\in \Pro(\mathcal{C})$ be the
pro-object representing the left exact functor $$\mathcal{C}\to
\Vect,\,X\mapsto \Hom_{\mathcal{M}}(M_2,X\ot^{\mathcal{M}} M_1),$$ i.e.,
$$\Hom_{\mathcal{M}}(M_2,X\ot^{\mathcal{M}} M_1)
\cong\Hom_{\text{Pro}(\mathcal{C})} (\overline{\Hom}(M_1,M_2),X).$$
For any $M\in \mathcal{M}$, the pro-object
$\overline{\Hom}(M,M)$ has a canonical
structure of a coalgebra. If $\mathcal{M}$ is {\em indecomposable} (i.e., $\mathcal{M}$ is not equivalent to a direct sum of two nontrivial module subcategories) and {\em exact} (see definition below) then the category
$\Comod_{\Pro(\mathcal{C})}(\overline{\Hom}(M,M))$ of right
comodules over $\overline{\Hom}(M,M)$ in $\Pro(\mathcal{C})$,
equipped with its canonical structure of a $\mathcal{C}-$module
category, is equivalent to $\mathcal{M}$. (This is a special case of Barr-Beck Theorem in category theory; see \cite[Theorem 3.17]{EO}.). We note that in terms of
internal Hom's \cite{EO}, the algebra $\underline{\Hom}(M,M)$
in $\Ind(\mathcal{C})$ is isomorphic to the dual algebra of the coalgebra
$\overline{\Hom}(M,M)$ under the duality functor $^*:\Pro(\mathcal{C})\to\Ind(\mathcal{C})$.

\begin{example}\label{hopf}
Let $H$ be a Hopf algebra over $k$, let $\C:=\Rep(H)$ be the rigid tensor category of finite dimensional representations of $H$ over $k$, let $\mathcal{M}:=\Vect$ be the module category over $\C$ with respect to the forgetful functor $\C\to \Vect$ (equipped with the usual tensor structure), and let $\delta:=k$ be the trivial representation. Then $\underline{\Hom}(\delta,\delta)=H^{\circ}$ is the {\em finite dual} of $H$ (i.e., the Hopf algebra of linear functionals on $H$ vanishing on a finite codimensional ideal of $H$). Indeed, let $X\in \C$ and denote its underlying vector space by $\overline{X}$. On one hand, $$\Hom_{\Ind(\mathcal{C})}(X,\underline{\Hom}(k,k))=\Hom_{\Vect}(X\ot k,k)=\Hom_{\Vect}(\overline{X},k)=
\overline{X}^*.$$ On the other hand, since $X$ is finite dimensional, we have that 
\begin{eqnarray*}
\lefteqn{\Hom_{\Ind(\mathcal{C})}(X,H^{\circ})=\Hom_{\Ind(\mathcal{C})}(X,H^{*})=\Hom_{\Ind(\mathcal{C})}(X\ot H,k)}\\
& & =\Hom_{\Ind(\mathcal{C})}(\overline{X}\ot _k H,k)=\Hom_{\Ind(\mathcal{C})}(H,\overline{X}^*)=\overline{X}^*.
\end{eqnarray*}
Therefore the claim follows from Yoneda's lemma.

Consequently, $\overline{\Hom}(\delta,\delta)=(H^{\circ})^*=\widehat{H}$ is the profinite completion of $H$ with respect to its algebra structure (i.e., $\widehat{H}$ is the inverse limit $\varprojlim H/I$ over all finite codimensional proper ideals $I$ of $H$).  
\end{example}

\begin{definition}(cf. \cite{EO})
A module category $\mathcal{M}$ over
$\mathcal{C}$ is called \emph{exact} if any additive module functor
$\mathcal{M}\to \mathcal{M}_1$ from $\mathcal{M}$ to any other
$\mathcal{C}-$module category $\mathcal{M}_1$ is exact.
\end{definition}

\begin{remark}\label{2cat}
The collection of all exact $\C-$module categories forms a $2-$category $\Mod(\C)$: the $1-$morphisms in $\Mod(\C)$ are $\C-$module functors, and the $2-$morphisms are natural transformations of $\C-$module functors.
\end{remark}

\begin{proposition}\label{exact} (cf. \cite{EO})
Let $\mathcal{C}$ be a rigid tensor category over $k$, and let $\mathcal{M}$ be a module category over $\mathcal{C}$. Then the following are equivalent:

1) $\mathcal{M}$ is exact.

2) For any $M\in \mathcal{M}$
and any injective object $I\in \Ind(\mathcal{C})$, $I\ot M$ is injective in $\Ind(\mathcal{M})$.

3) For any $M\in \mathcal{M}$ and any projective object $P\in \Pro(\mathcal{C})$, $P\ot M$ is projective in $\Pro(\mathcal{M})$.
\end{proposition}

\begin{proof}
The proof that 1) implies 2) is exactly as
the proof of \cite[Proposition 3.16]{EO} (after replacing
``projective" by ``injective" and $``\underline{\Hom}"$ by
$``\overline{\Hom}"$). More precisely, if $\mathcal{M}$ is exact then the $\mathcal{C}-$module functor
$\overline{Hom}(M,?):\mathcal{M}\to \Pro(\mathcal{C})$ is exact. Therefore the functor $$\Hom_{\Ind(\mathcal{M})}(?,I\ot M)=\Hom_{\Ind(\mathcal{C})}(\overline{Hom}(M,?),I)$$ is exact for any injective object $I$ in $\Ind(\mathcal{C})$, so $I\ot M$ is injective in $\Ind(\mathcal{M})$. (Here by $\Hom_{\Ind(\mathcal{C})}(\overline{Hom}(M,?),I)$ we mean the $\Hom-$space $\Hom_{\Ind(\mathcal{C})}(\mathbf{1},\overline{Hom}(M,?)^*\otimes I)$.)

The proof that 2) implies 1) is exactly as the proof of
\cite[Proposition 3.11]{EO} (after replacing ``projective" by
``injective").

Finally, 2) is equivalent to 3) by duality.
\end{proof}

Let $\mathcal{C}$ be a rigid tensor category over $k$, and let
$\mathcal{M}$ be a module category over $\mathcal{C}$. Following
\cite{EO}, we say that two simple objects $M_1,M_2\in \mathcal{M}$
are \textit{related} if there exists an object $X\in \mathcal{C}$ such that
$M_1$ appears as a subquotient in $X\ot M_2$.

\begin{proposition} (\cite{EO})
Let $\mathcal{C}$ be a rigid tensor category over $k$ and let
$\mathcal{M}$ be an exact module category over $\mathcal{C}$. Then
the following hold:

1) The above relation is an equivalence relation.

2) $\mathcal{M}$
decomposes into a direct sum $\mathcal{M}=\oplus \mathcal{M}_i$ of
indecomposable exact module subcategories indexed by the equivalence
classes of the above relation.
\end{proposition}

\begin{proof}
1) The proof is essentially the proof of \cite[Lemma 3.8 and
Proposition 3.9]{EO}. Namely, the proof that the relation is
reflexive and transitive is exactly the same. Suppose that $M_1$
appears as a subquotient in $X\ot M_2$, and let
$E(\textbf{1})\in\Ind(\mathcal{C})$ be the injective hull of the
unit object $\textbf{1}\in \mathcal{C}$. By Proposition \ref{exact}, $E(\textbf{1})\ot X\ot M_2$ is injective in $\Ind(\mathcal{M})$, and hence
$$\Hom_{\Pro(\mathcal{M})}(X^*\ot E(\textbf{1})^*\ot M_1,M_2)=
\Hom_{\Ind(\mathcal{M})}(M_1,E(\textbf{1})\ot X\ot M_2)\ne 0.$$ This implies the existences of
$Y\in\mathcal{C}$ such that $\Hom_{\mathcal{M}}(
Y\ot M_1,M_2)\ne 0$, which proves that
the relation is also symmetric.

2) For an equivalence class $i$ let $\mathcal{M}_i$ be the
full subcategory of $\mathcal{M}$ consisting of objects all simple
subquotients of which lie in $i$. Clearly, $\mathcal{M}_i$ is an
indecomposable module subcategory of $\mathcal{M}$ and
$\mathcal{M}=\oplus \mathcal{M}_i$. Furthermore, $\mathcal{M}_i$ is
exact since so is $\mathcal{M}$.
\end{proof}

\begin{definition} (cf. \cite{EO})
Let $\mathcal{M}$ be an exact module category over $\mathcal{C}$. We say that an object $\delta\in\mathcal{M}$ generates $\mathcal{M}$ if for
any $M\in \mathcal{M}$ there exists $X\in\mathcal{C}$ such that
$\Hom_{\mathcal{M}}(X\ot ^{\mathcal{M}}\delta,M)\ne 0$. 
\end{definition}

\begin{remark} (cf. \cite{EO})
It is known that $\delta$ generates $\mathcal{M}$ if and only if its simple
subquotients represent all equivalence classes of simple objects in $\mathcal{M}$. Therefore, $\delta$ generates $\mathcal{M}$ if and only if for any $M\in \mathcal{M}$ there exists $X\in\mathcal{C}$ such that
$M$ is a subquotient of $X\ot ^{\mathcal{M}}\delta$.
\end{remark}

\begin{corollary}\label{simgen}
Let $\mathcal{M}$ be an indecomposable exact module category over a
rigid tensor category $\mathcal{C}$, and let $\delta\in \mathcal{M}$
be a simple object. Then $\delta$ generates $\mathcal{M}$. \qed
\end{corollary}

\section{exact module categories over $\text{Coh}_f(G)$}\label{cohg}

Let $G$ be an affine group scheme over $k$ (see Section 2.1).

\subsection{The category $\text{Coh}_f(G)$}

We shall denote by $\text{Coh}_f(G)$ (resp., $\text{Coh}(G)$) the abelian category of coherent sheaves of \emph{finite dimensional} $k-$vector spaces on $G$, i.e., coherent sheaves
supported on finite sets in $G$ (resp., \emph{all} coherent sheaves of $k-$vector spaces on $G$). Recall that $\text{Coh}_f(G)$ (resp., $\text{Coh}(G)$) is a rigid tensor category (resp., tensor category) with the convolution product
$$X\otimes Y:=m_*(X\boxtimes Y)$$ as the tensor product
(where $m_*$ is the direct image functor of $m$). It is known that $\text{Coh}_f(G)$ (resp., $\text{Coh}(G)$) is tensor equivalent to the rigid tensor category $\Rep(\mathcal{O}(G))$ of \emph{finite dimensional}
$k-$representations of the Hopf algebra $\mathcal{O}(G)$ (resp., the tensor category of \emph{finitely generated}
$k-$representations of the Hopf algebra $\mathcal{O}(G)$). Recall also that $\Ind(\text{Coh}_f(G))$ is the category of locally finite representations of $\mathcal{O}(G)$, i.e., representations in which every vector generates a finite dimensional subrepresentation, while $\Ind(\text{Coh}(G))=\text{QCoh}(G)$ is the category of quasicoherent sheaves of $k-$vector spaces on $G$, i.e., the category of \emph{all}
$k-$representations of the Hopf algebra $\mathcal{O}(G)$.

Let $\delta_g$ denote the {\em simple} object in $\text{Coh}_f(G)$ corresponding to the closed point $g\in G$. It is well known that $\Ext^1_{\mathcal{O}(G)}(\delta_g,\delta_h)=0$ if $g\ne h$, and hence $\text{Coh}_f(G)$ decomposes, as an abelian category, into a direct sum $\text{Coh}_f(G)=\oplus_{g\in G}\text{Coh}_f(G)_g$, where $\text{Coh}_f(G)_g$ is the abelian subcategory of sheaves supported at $g$. Since each of these subcategories has a unique simple object, there is a unique {\em indecomposable projective} object in the pro-completion category, which is $P_g:=\widehat{\mathcal{O}(G)_g}$ (= the completion of $\mathcal{O}(G)$ at $g$).
Therefore any projective in $\Pro(\text{Coh}_f(G))$ is a completed direct sum of such $P_g$.

\begin{example}
If $G$ is a finite abstract group, i.e., a constant group scheme over $k$, then $\text{Coh}_f(G)=\text{Coh}(G)$ is nothing but the fusion category $\Vect(G)$ of finite dimensional $G-$graded vector spaces over $k$.
\end{example}

\subsection{Equivariant coherent sheaves}

Let $H$ be a closed group subscheme of $G$ (see Section 2.1), and let $\mu
:G\times H\to G$ be its free action on $G$ by right translations (in other words, the free actions of $H(R)$ on $G(R)$ by right translations are functorial in $R$, $R$ a commutative algebra over $k$). Set
$$\eta:=\mu(id\times m_{|H})=\mu(\mu\times id):
G\times H\times H\to G.$$ Let
$$\text{p}_{GH}^1:G\times H\to G,\,\text{p}_{GHH}^1:G\times H\times H\to G,\,
\text{p}_{GHH}^{12}:G\times H\times H\to G\times H$$ be the projections on
$G$, $G$ and $G\times H$, respectively. We clearly have that $p_{GH}^1\circ
p_{GHH}^{12}=p_{GHH}^1$.

Let $\psi: H\times H\to \mathbb{G}_m$ be a normalized $2-$cocycle.
Equivalently, $\psi\in \mathcal{O}(H)\ot \mathcal{O}(H)$ is a
\emph{Drinfeld twist} for $\mathcal{O}(H)$, i.e., $\psi$ is an
invertible element satisfying the equations
$$(\Delta\ot id)(\psi)(\psi\ot 1)=(id\ot \Delta)(\psi)(1\ot \psi),$$
$$(\varepsilon\ot id)(\psi)=(id\ot \varepsilon)(\psi)=1.$$ We let
$\mathcal{O}(H)_{\psi}$ be the (``twisted") coalgebra with underlying vector
space $\mathcal{O}(H)$ and comultiplication $\Delta_{\psi}$ given by
$\Delta_{\psi}(f):=\Delta(f)\psi$, where $\Delta$ is the standard
comultiplication of $\mathcal{O}(H)$.

Note that $\psi$ (like any other regular non-vanishing function)
defines an automorphism of any coherent sheaf on $H\times H$ by
multiplication.

\begin{definition}\label{defequiv}
Let $\psi: H\times H\to \mathbb{G}_m$ be a normalized $2-$cocycle on
a closed group subscheme $H$ of $G$.

1) An \emph{$(H,\psi)-$equivariant coherent sheaf} on $G$ is a pair
$(S,\lambda_S)$, where $S\in \text{Coh}(G)$ and $\lambda_S$ is an
isomorphism $\lambda_S:(\text{p}_{GH}^1)^*(S)\to \mu^*(S)$ of sheaves on
$G\times H$ such that the diagram of morphisms of sheaves on
$G\times H\times H$
\begin{equation*}
\label{equivariantX} \xymatrix{(\text{p}_{GHH}^1)^*(S) \ar[d]_{(id\times
m_{|H})^*(\lambda_S)}\ar[rr]^{(\text{p}_{GHH}^{12})^*(\lambda_S)}
&& \mu^*(S)\ar[d]^{(\mu\times id)^*(\lambda_S)} &&\\
\eta^*(S)\ar[rr]_{id\boxtimes \psi} && \eta^*(S)&&}
\end{equation*}
is commutative.

2) Let $(S,\lambda_S)$ and $(T,\lambda_T)$ be two
$(H,\psi)-$equivariant coherent sheaves on $G$. A morphism $\phi:S\to T$ in
$\text{Coh}(G)$ is said to be \emph{$(H,\psi)-$equivariant} if the
diagram of morphisms of sheaves on $G\times H$
\begin{equation*}
\label{equivariantX} \xymatrix{(p_{GH}^1)^*(S)
\ar[d]_{\lambda_S}\ar[rr]^{(p_{GH}^1)^*(\phi)}
&& (p_{GH}^1)^*(T)\ar[d]^{\lambda_T} &&\\
\mu^*(S)\ar[rr]_{\mu^*(\phi)} && \mu^*(T)&&}
\end{equation*}
is commutative.

3) Let $\text{Coh}_f^{(H,\psi)}(G)$ be the abelian category of
$(H,\psi)-$equivariant coherent sheaves on $G$ \emph{with finite support in
$G/H$} (i.e., sheaves supported on {\em finitely many} $H-$cosets), with
$(H,\psi)-$equivariant morphisms.
\end{definition}

\begin{example}\label{ex1}
$\text{Coh}_f^{(\{1\},1)}(G)=\text{Coh}_f(G)$, and $\text{Coh}_f^{(G,\psi)}(G)=\Vect$ (the unique simple object being the regular representation $\mathcal{O}(G)$).
\end{example}

\begin{remark}\label{biequiv}
Let $(H',\psi')$ be another pair consisting of a closed group subscheme $H'$ of $G$ and a normalized $2-$cocycle $\psi'$ on it. By considering the free right action of $H'\times H$ on $G$ given by $g(a,b):=a^{-1}gb$, we can similarly
define {\em $((H',\psi'),(H,\psi))-$biequivariant coherent sheaves} on $G$, and the abelian category $\text{Coh}_f^{((H',\psi'),(H,\psi))}(G)$.
\end{remark}

\begin{remark}
1) In \cite[p.110]{Mum}, an $H-$equivariant sheaf $(S,\lambda_S)$ on $G$ is referred to as a
sheaf $S$ on $G$ together with a \emph{lift} $\lambda_S$ of the $H-$action $\mu$ on $G$ to $S$.

2) It is well known that the (geometric, hence also categorical)
quotient scheme $G/H$ exists. Let $\text{Coh}_f(G/H)$ be the
abelian category of coherent sheaves of \emph{finite dimensional}
$k-$vector spaces on $G/H$, and let $\pi:G\to G/H$ be the
canonical $H-$invariant morphism. It is known that the inverse
image functor $\pi^*:\text{Coh}_f(G/H)\to \text{Coh}(G)$ determines an
equivalence of categories between $\text{Coh}_f(G/H)$ and
$\text{Coh}_f^{(H,1)}(G)$, with $\pi_*^H$ as its inverse (where $\pi_*^H$
is the subsheaf of $H-$invariants of $\pi_*$).
\end{remark}

This following lemma will be very useful in the sequel.

\begin{lemma}\label{simpleex}
Let $H$ be a closed group subscheme of an affine group scheme $G$
over $k$, acting on itself and on $G$ by right translations
$\mu_H:H\times H\to H$ and $\mu_G:G\times H\to G$, respectively. Let
$\iota=\iota_H:H\hookrightarrow G$ be the inclusion morphism, and let $\psi$ be a normalized $2-$cocycle on $H$. The following hold:

1) The structure sheaf $\mathcal{O}(H)$ of $H$ admits a canonical structure of an $(H,\psi)-$equivariant coherent sheaf on $H$, making it the unique (up to isomorphism) simple object
of $\text{Coh}_f^{(H,\psi)}(H)$.

2) The sheaf $\iota_*\mathcal{O}(H)\in \text{Coh}(G)$ (i.e., the
representation of $\mathcal{O}(G)$ on $\mathcal{O}(H)$ coming from the morphism $\iota$) is a simple object in
$\text{Coh}_f^{(H,\psi)}(G)$.

3) For any $X\in \text{Coh}_f(G)$ and $M\in \text{Coh}_f^{(H,\psi)}(G)$, we have that
$m_*(X\boxtimes M)\in \text{Coh}_f^{(H,\psi)}(G)$.
\end{lemma}

\begin{proof}
1) Consider the isomorphism $\varphi:=(\mu_H,p_{HH}^2):H\times
H\xrightarrow{\cong} H\times H$, where $p_{HH}^2:H\times H\to H$ is the projection on the second coordinate. Clearly, $p_{HH}^1\circ\varphi=\mu_H$,
so $(p_{HH}^1\circ\varphi)^*\mathcal{O}(H)=\mu_H^*\mathcal{O}(H)$. Now,
multiplication by $\psi$ defines an isomorphism
$$\mu_H^*\mathcal{O}(H)=(p_{HH}^1\circ\varphi)^*\mathcal{O}(H)\xrightarrow{\psi}
(\varphi^*\circ (p_{HH}^1)^*)\mathcal{O}(H),$$ and since we have that
$(p_{HH}^1)^*\mathcal{O}(H)=\mathcal{O}(H)\ot \mathcal{O}(H)$, we get an
isomorphism $$\lambda:(p_{HH}^1)^*\mathcal{O}(H)\xrightarrow{=}
\mathcal{O}(H)\ot \mathcal{O}(H)\xrightarrow{\varphi^*}
\varphi^*(\mathcal{O}(H)\ot \mathcal{O}(H))\xrightarrow{\psi^{-1}}
\mu_H^*\mathcal{O}(H).$$ The fact that
$(\mathcal{O}(H),\lambda)$ is an $(H,\psi)-$equivariant coherent sheaf on $H$ can be checked now in a straightforward manner. Clearly, $(\mathcal{O}(H),\lambda)$ is the unique (up to isomorphism) simple object of $\text{Coh}_f^{(H,\psi)}(H)$ (see Example \ref{ex1}).

2) Since $\iota$ is affine, the commutative diagrams
\begin{equation*}
\xymatrix{H\times H \ar[d]_{\iota \times id_H}\ar[rr]^{p_{HH}^1}
&& H \ar[d]^{\iota} &&\\
G\times H\ar[rr]_{p_{GH}^1} && G&&} \xymatrix{H\times H \ar[d]_{\iota
\times id_H}\ar[rr]^{\mu_H}
&& H \ar[d]^{\iota} &&\\
G\times H\ar[rr]_{\mu_G} && G&&}
\end{equation*}
yield isomorphisms 
\begin{equation}\label{eq1}
(p_{GH}^1)^*\iota_*\mathcal{O}(H)\xrightarrow{\cong}
(\iota \times id_H)_*(p_{HH}^1)^* \mathcal{O}(H)
\end{equation} 
and
\begin{equation}\label{eq2}
(\iota \times id_H)_*\mu_H^* \mathcal{O}(H) \xrightarrow{\cong}
\mu_G^*\iota_*\mathcal{O}(H)
\end{equation} 
(``base change").

Let $\lambda:(p_{HH}^1)^*\mathcal{O}(H)\xrightarrow{\cong}
\mu_H^*\mathcal{O}(H)$ be the isomorphism constructed in Part 1. Since $\iota$ is $H-$equivariant, we get an isomorphism
\begin{equation}\label{eq3}
(\iota \times
id_H)_*(p_{HH}^1)^* \mathcal{O}(H) \xrightarrow{(\iota \times
id_H)_*(\lambda)} (\iota \times id_H)_*\mu_H^* \mathcal{O}(H).
\end{equation}
It is now straightforward to check that the composition of isomorphisms (\ref{eq1}), (\ref{eq3}) and (\ref{eq2})
\begin{equation*}
(p_{GH}^1)^*\iota_*\mathcal{O}(H) \xrightarrow{\cong} 
\mu_G^*\iota_*\mathcal{O}(H)
\end{equation*}
endows $\iota_*\mathcal{O}(H)$ with a structure of an $(H,\psi)-$equivariant coherent sheaf on $G$. Clearly, $\iota_*\mathcal{O}(H)$ is simple.

3) Consider the right action $id\times \mu:G\times G\times H\to
G\times G$ of $H$ on $G\times G$. Since $M\in
\text{Coh}_f^{(H,\psi)}(G)$ it is clear that $X\boxtimes M\in \text{Coh}(G\times G)$ is an $(H,\psi)-$equivariant coherent sheaf on $G\times G$ (here we identify $H$ with the subgroup $\{1\}\times H\subseteq G\times G$). But since $m:G\times G\to G$ is $H-$equivariant, $m_*$
carries $(H,\psi)-$equivariant coherent sheaves on $G\times G$ to
$(H,\psi)-$equivariant coherent sheaves on $G$.
\end{proof}

\subsection{Exact module categories}

Let $G$, $H$ and $\psi$ be as in
3.1--3.2, and consider the coalgebra $\mathcal{O}(H)_{\psi}$
in $\text{Coh}(G)$. Let $\widehat{\mathcal{O}(H)_{\psi}}$ be its profinite completion with respect to the algebra structure of $\mathcal{O}(H)$ (see Example \ref{hopf}); it is a coalgebra object in both $\Pro(\text{Coh}(G))$ and $\Pro(\text{Coh}_f(G))$. Let
$\Comod_{\Pro(\text{Coh}_f(G))}(\widehat{\mathcal{O}(H)_{\psi}})$ be the abelian
category of right comodules over $\widehat{\mathcal{O}(H)_{\psi}}$ in
$\Pro(\text{Coh}_f(G))$.

\begin{proposition}\label{grsch}
Let $G$ be an affine group scheme over $k$, let $H$
be a closed group subscheme of $G$, let $\psi$ be a normalized
$2-$cocycle on $H$, and let $\delta=\delta_{(H,\psi)}:=(\iota_H)_*\mathcal{O}(H)\in \text{Coh}_f^{(H,\psi)}(G)$. The following hold:

1) Set $\mathcal{M}:=\text{Coh}_f^{(H,\psi)}(G)$. The bifunctor
$$\ot^{\mathcal{M}}:\text{Coh}_f(G)\boxtimes\mathcal{M}\to
\mathcal{M},\,\, X\boxtimes M\mapsto m_*(X\boxtimes M),$$ defines on
$\mathcal{M}$ a structure of an indecomposable
$\text{Coh}_f(G)-$module category.

2) Set $\mathcal{V}:=\Comod_{\Pro(\text{Coh}_f(G))}(\widehat{\mathcal{O}(H)_{\psi}})$.
The bifunctor
$$\ot^{\mathcal{V}}:\text{Coh}_f(G)\boxtimes\mathcal{V}
\to \mathcal{V},\,\, X\boxtimes V\mapsto m_*(X\boxtimes V),$$
defines on $\mathcal{V}$ a structure of an $\text{Coh}_f(G)-$module
category.

3) The categories $\text{Coh}_f^{(H,\psi)}(G)$ and
$\Comod_{\Pro(\text{Coh}_f(G))}(\widehat{\mathcal{O}(H)_{\psi}})$ are equivalent as
module categories over $\text{Coh}_f(G)$. In particular, $\overline{\Hom}(\delta,\delta)$ and $\widehat{\mathcal{O}(H)_{\psi}}$ are isomorphic as coalgebras in $\Pro(\text{Coh}_f(G))$.
\end{proposition}

\begin{proof}
1) Since $m(m\times id)=m(id\times m)$ and $\psi$ is a $2-$cocycle,
it follows from Lemma \ref{simpleex} that $\ot^{\mathcal{M}}$
defines on $\mathcal{M}$ a structure of an $\text{Coh}_f(G)-$module
category. Clearly, $\text{Coh}_f(H)$ is the subcategory of $\text{Coh}_f(G)$ consisting of those objects $X$ for which $X\ot ^{\mathcal{M}} \delta $ is a multiple of $\delta$, and any object $M\in\mathcal{M}$ is of the form $X\ot ^{\mathcal{M}}\delta$ for some $X\in \text{Coh}_f(G)$. In particular, the simple object $\delta$ (see Lemma \ref{simpleex}) generates $\mathcal{M}$, so $\mathcal{M}$ is indecomposable.

2) By definition, an object in $\mathcal{V}$ is a pair $(V,\rho_V)$ consisting of an object
$V\in \Pro(\text{Coh}_f(G))$
and a morphism $\rho_V:V\to V\ot\widehat{\mathcal{O}(H)_{\psi}}$ in
$\Pro(\text{Coh}_f(G))$ satisfying the comodule axioms. Clearly, if
$X\in\text{Coh}_f(G)$ then $m_*(X\boxtimes V)\in\Pro(\text{Coh}_f(G))$, and $\rho_{m_*(X\boxtimes
V)}:=id_X\ot\rho_V$ is a morphism in $\Pro(\text{Coh}_f(G))$ defining on
$m_*(X\boxtimes V)$ a structure of a right comodule over
$\widehat{\mathcal{O}(H)_{\psi}}$.

3) For any $S\in \Pro(\text{Coh}_f(G))$ there is a natural isomorphism $$\Hom_{G\times
H}(\mu^*(S),(p_{GH}^1)^*(S))\cong \Hom_G(S,\mu_*(p_{GH}^1)^*(S))$$ (``adjunction"). Since $\mu_*(p_{GH}^1)^*(S)\cong S\ot
\widehat{\mathcal{O}(H)}$, we can assign
to any isomorphism $\lambda:\mu^*(S)\to (p_{GH}^1)^*(S)$ a morphism
$\rho_{\lambda}:S\to S\ot \widehat{\mathcal{O}(H)}$. It is now straightforward
to verify that $(S,\lambda^{-1})$ is an $(H,\psi)-$equivariant sheaf
on $G$ if and only if $\rho_{\lambda}:S\to S\ot
\widehat{\mathcal{O}(H)_{\psi}}$ is a comodule map.
\end{proof}

\begin{example}
We have that $\text{Coh}_f^{(\{1\},1)}(G)=\text{Coh}_f(G)$ is the regular module, and $\text{Coh}_f^{(G,1)}(G)=\Vect$ is the usual fiber functor on $\text{Coh}_f(G)$.
\end{example}

We are now ready to state and prove the main result of
this section.

\begin{theorem}\label{grsch1}
Let $G$ be an affine group scheme over $k$. There is
a bijection between conjugacy classes of pairs $(H,\psi)$ and
equivalence classes of indecomposable exact module categories over
$\text{Coh}_f(G)$, assigning $(H,\psi)$ to $\text{Coh}_f^{(H,\psi)}(G)$.
\end{theorem}

\begin{proof}
We first show that the indecomposable $\text{Coh}_f(G)-$module category $\mathcal{M}:=\text{Coh}_f^{(H,\psi)}(G)$ is exact. To this end, we have to show that if $P\in\Pro(\text{Coh}_f(G))$ is projective and $X\in \mathcal{M}$, then $P\ot ^{\mathcal{M}} X$ is projective (see Proposition \ref{exact}). Clearly, it is sufficient to show it for $X:=\delta=\delta_{(H,\psi)}$. Moreover, since any projective in $\Pro(\text{Coh}_f(G))$ is a completed direct sum of $P_g$ (see Section 3.1), it suffices to check that $P_g\ot ^{\mathcal{M}} \delta$ is projective. Furthermore, since $P_g=\delta_g\otimes P_1$, and $\delta_g\ot ^{\mathcal{M}} \,?$ is an autoequivalence of $\mathcal{M}$ as an abelian category (since $\delta_g$ is invertible), it suffices to do so for $g=1$. Finally, this is done just by computing this product explicitly using the definition, which yields that $P_1\ot ^{\mathcal{M}} \delta=\widehat{\mathcal{O}(H)_1}\otimes_k P(\delta)$, where $P(\delta)$ is the projective cover of $\delta$ (i.e., the unique indecomposable projective in the block of $\Pro(\mathcal{M})$ containing $\delta$; as a sheaf on $G$, it is the function algebra on the formal neighborhood of $H$), and hence projective as desired. 

Conversely, we have to show that any indecomposable exact module
category $\mathcal{M}$ over $\text{Coh}_f(G)$ is of the form
$\text{Coh}_f^{(H,\psi)}(G)$. Indeed, let $\delta\in \mathcal{M}$ be a
simple object generating $\mathcal{M}$ (such $\delta$ exists by
Corollary \ref{simgen}), and consider the full subcategory
$$\mathcal{C}:=\{X\in \text{Coh}_f(G)\mid X\ot ^{\mathcal{M}}
\delta=\dim_k(X)\delta\}$$ of $\text{Coh}_f(G)$. Clearly,
$\mathcal{C}$ is a tensor subcategory of $\text{Coh}_f(G)$. Therefore, there exists a closed group subscheme $H$ of $G$ such that
$\mathcal{C}\cong \text{Coh}_f(H)$ as tensor categories, and we may identify $\mathcal{C}$ with $\text{Coh}_f(H)$. Moreover, the functor
$$F:\mathcal{C}=\text{Coh}_f(H)\to \Vect,\,F(X)=\Hom_{\mathcal{M}}(\delta,X\ot ^{\mathcal{M}} \delta),$$ together with the tensor structure
$F(\cdot)\otimes F(\cdot)\xrightarrow{\cong} F(\cdot\otimes \cdot)$
coming from the associativity constraint, is a fiber functor on $\text{Coh}_f(H)$. But, letting $\overline{X}$ denote the underlying vector space of $X$ (where we view $X$ as an $\mathcal{O}(H)-$module),
we see that $F(X)=\overline{X}$. We therefore get a functorial isomorphism
$\overline{X}\otimes \overline{Y}\xrightarrow{\cong} \overline{X\otimes Y}$, which is nothing but an
invertible element $\psi$ of $\mathcal{O}(H)\otimes
\mathcal{O}(H)=\mathcal{O}(H\otimes H)$ taking values in
$\mathbb{G}_m(k)$. Clearly, $\psi$ is a twist for $\mathcal{O}(H)$.
To summarize, we have obtained that the $\C-$submodule category $<\delta>$ of $\mathcal{M}$ consisting of all multiples of $\delta$ is equivalent to $\text{Coh}_f^{(H,\psi)}(H)$ as a module category over $\mathcal{C}=\text{Coh}_f(H)$.

Finally, let $X\in \text{Coh}_f(G)$, and let $X_H\in \text{Coh}_f(H)$ be the maximal subsheaf of $X$ which is scheme theoretically supported on $H$ (i.e., $\overline{X_H}$ consists of all vectors in $\overline{X}$ which are annihilated by the defining ideal of $H$ in $\mathcal{O}(G)$). Now, on the one hand, since for any $g\in G$, $\delta_g\ot ^{\mathcal{M}}\delta$ is simple, and $\delta_g\ot ^{\mathcal{M}}\delta\cong\delta$ if and only if $g\in H$, it is clear that
$$\Hom_{\Pro(\text{Coh}_f(G))}
(\overline{\Hom}(\delta,\delta),X)=\Hom_{\mathcal{M}}(\delta,X\ot
^{\mathcal{M}}\delta)=\overline{X_H}$$ 
(since it holds for any simple $X$). On the other hand, it is clear that $$\Hom_{\Pro(\text{Coh}_f(G))}
(\widehat{\mathcal{O}(H)_{\psi}},X)=\overline{X_H}.$$
Therefore by Yoneda's lemma, the two coalgebras $\overline{\Hom}(\delta,\delta)$ and $\widehat{\mathcal{O}(H)_{\psi}}$ in $\Pro(\text{Coh}_f(G))$ are isomorphic. But this implies that $\mathcal{M}$ is equivalent
to $\Comod_{\Pro(\text{Coh}_f(G))}(\widehat{\mathcal{O}(H)_{\psi}})$ as a module category over $\text{Coh}_f(G)$ (as $\mathcal{M}$ is
indecomposable, exact, and generated by $\delta$), hence also to $\text{Coh}_f^{(H,\psi)}(G)$ by Proposition \ref{grsch}. 

We are done.
\end{proof}

\section{Exact module categories over $\Rep(G)$}\label{repg}

Let $\mathcal{C}$ be a rigid tensor category. Given two exact module categories $\mathcal{M}$, $\mathcal{N}$ over $\mathcal{C}$, let $\Fun_{\C}(\mathcal{M},\mathcal{N})$ denote the abelian category of $\mathcal{C}-$functors from $\mathcal{M}$ to $\mathcal{N}$. The \emph{dual category of $\mathcal{C}$ with respect to $\mathcal{M}$} is the category $\mathcal{C}^*_{\mathcal{M}}:=\End_{\mathcal{C}}(\mathcal{M})$ of $\mathcal{C}-$endofunctors of $\mathcal{M}$. If $\mathcal{M}$ is indecomposable, $\mathcal{C}^*_{\mathcal{M}}$ is a rigid tensor category, and $\mathcal{M}$ is an indecomposable exact module category over $\mathcal{C}^*_{\mathcal{M}}$. Also, $\Fun_{\C}(\mathcal{M},\mathcal{N})$ is an exact module category over
$\mathcal{C}^*_{\mathcal{M}}$ via the composition of functors.

\subsection{Module categories.} 

We keep the notation from Section \ref{cohg}, and set ${\mathcal{M}(H,\psi)}:=\text{Coh}_f^{(H,\psi)}(G)$; in particular, $\mathcal{M}(G,1)=\Vect$ is the usual fiber functor on $\text{Coh}_f(G)$. We also set ${\mathcal{M}(G,(H,\psi))}:=\text{Coh}_f^{(G,(H,\psi))}(G)$ (see Remark \ref{biequiv}). 

Recall that the $2-$cocycle $\psi$ determines a central extension
$H_{\psi}$ of $H$ by $\mathbb{G}_m$. By an {\em $(H,\psi)-$representation} of $H$ we shall mean a rational representation of the group
scheme $H_{\psi}$ on which $\mathbb{G}_m$ acts with weight $1$
(i.e., via the identity character). Let us denote the category of finite dimensional $(H,\psi)-$representations of the group scheme $H_{\psi}$ by $\Rep(H,\psi)$. Clearly, $\Rep(H,\psi)$ is equivalent to the category $\Corep(\mathcal{O}(H)_{\psi})$ of finite dimensional comodules over the twisted coalgebra $\mathcal{O}(H)_{\psi}$ (see Section 3.2).

\begin{lemma}\label{modrep}
The following hold:

1) The category $\Fun_{\text{Coh}_f(G)}(\mathcal{M}(G,1),\mathcal{M}(H,\psi))$ is equivalent to the category ${\mathcal{M}(G,(H,\psi))}$ as an abelian category. In particular, $\Rep(G)$ is equivalent to $\text{Coh}_f(G)^*_{\mathcal{M}(G,1)}$ as a tensor category.

2) The categories
$\Fun_{\text{Coh}_f(G)}(\mathcal{M}(G,1),\mathcal{M}(H,\psi))$ and $\Rep(H,\psi)$ 
are equivalent as module categories over $\Rep(G)$.
\end{lemma}

\begin{proof}
1) Since $\mathcal{M}(G,1)=\Vect$, a functor $\mathcal{M}(G,1)\to\mathcal{M}(H,\psi)$ is
just an $(H,\psi)-$equivariant sheaf $X$ on $G$. The fact that the
functor is a $\text{Coh}_f(G)-$module functor gives $X$ a commuting
$G-$equivariant structure for the left action of $G$ on itself,
i.e., $X$ is $(G,(H,\psi))-$biequivariant. Conversely, it is clear that any $(G,(H,\psi))-$biequivariant sheaf on $G$ defines a $\text{Coh}_f(G)-$module functor $\mathcal{M}(G,1)\to\mathcal{M}(H,\psi)$.

Finally, the category of $(G,G)-$biequivariant sheaves
on $G$ is equivalent to the category $\Rep(G)$ as a tensor category, and the second claim follows.

2) By Part 1), we may identify $\Fun_{\text{Coh}_f(G)}(\mathcal{M}(G,1),\mathcal{M}(H,\psi))$ and ${\mathcal{M}(G,(H,\psi))}$ as abelian categories.

Now, if $X$ is a $(G,(H,\psi))-$biequivariant sheaf on $G$ then the
inverse image sheaf $e^{*}(X)$ on $\Spec(k)$ (``the stalk at $1$")
acquires a structure of an $(H,\psi)-$representation via the action
of the element $(h,h^{-1})$ in $G\times H$, i.e., it is an object in
$\Rep(H,\psi)$. We have thus defined a functor $${\mathcal{M}(G,(H,\psi))}\to \Rep(H,\psi),\,\,X\mapsto e^{*}(X).$$

Conversely, an $(H,\psi)-$representation $V$ can be spread out over
$G$ and made into a $(G,(H,\psi))-$biequivariant sheaf $X$ on $G$, with global sections
$\mathcal{O}(G)\otimes_k V$. We have thus defined a functor $$\Rep(H,\psi)\to {\mathcal{M}(G,(H,\psi))},\,\,V\mapsto \mathcal{O}(G)\otimes_k V.$$

Finally, it is straightforward to verify that the two functors constructed above are mutually inverse.
\end{proof}

\begin{example}
The $\Rep(G)-$module category $\Rep(\{1\},1)=\Vect$ is the usual fiber functor on $\Rep(G)$. 
\end{example}

The proof of the next lemma is similar to the proof of Lemma \ref{modrep}.

\begin{lemma}\label{modfun1}
We have that $\text{Coh}_f(G)$ is equivalent to $\Rep(G)_{\Rep(\{1\},1)}^*$ as a tensor category, and the categories $\Fun_{\Rep(G)}(\Rep(\{1\},1),\Rep(H,\psi))$ and $\mathcal{M}(H,\psi)$ are equivalent as module categories over $\text{Coh}_f(G)$. \qed
\end{lemma}

Lemma \ref{modfun1} prompts the following definition.

\begin{definition}
A {\em geometrical} module category $\mathcal{N}$ over
$\Rep(G)$ is an exact module category $\mathcal{N}$ such that $\Fun_{\Rep(G)}(\Rep(\{1\},1),\mathcal{N}_i)\ne 0$ for any indecomposable direct summand module category $\mathcal{N}_i$ of $\mathcal{N}$.
\end{definition}

It is clear that geometrical module categories over
$\Rep(G)$ form a full $2-$subcategory $\Mod_{geom}(\Rep(G))$ of the $2-$category $\Mod(\Rep(G))$ (see Remark \ref{2cat}). 

We can now deduce from Lemmas \ref{modrep}, \ref{modfun1} the main result of this section, which says that geometrical module categories over $\Rep(G)$ are precisely those exact module categories which come from exact module categories over $\text{Coh}_f(G)$. More precisely, we have the following theorem, which generalizes a known result in the finite group case (see \cite{O}).

\begin{theorem}\label{modg}
Let $G$ be an affine group scheme over $k$. Then the $2-$functors
$$\Mod(\text{Coh}_f(G))\to \Mod_{geom}(\Rep(G)),\,\,
\mathcal{M}\mapsto\Fun_{\text{Coh}_f(G)}(\mathcal{M}(G,1),\mathcal{M}),$$ 
and
$$\Mod_{geom}(\Rep(G))\to \Mod(\text{Coh}_f(G)),\,\,
\mathcal{N}\mapsto\Fun_{\Rep(G)}(\Rep(\{1\},1),\mathcal{N}),$$
are $2-$equivalences which are mutually inverse. 
In particular, there is
a bijection between conjugacy classes of pairs $(H,\psi)$ and
equivalence classes of indecomposable geometrical module categories over
$\Rep(G)$, assigning $(H,\psi)$ to $\Rep(H,\psi)=\Corep(\mathcal{O}(H)_{\psi})$. \qed
\end{theorem}

\begin{remark}\label{notall} 1) If $G$ is \emph{not} finite, $\Rep(G)$ may very well have nongeometrical module categories (even fiber functors). For example, let $G:=\mathbb{G}_a^2$ over $\mathbb{C}$, let $J:=\exp(x\otimes
y)$, where $x,y$ are a basis of the (nilpotent) Lie algebra $\mathbb{C}^2$ (this
makes sense on $G-$modules, since on them $x,y$ are nilpotent so the
Taylor series for exponential terminates), and let $\mathcal{N}$ be
the semisimple $\Rep(G)-$module category of rank $1$ corresponding
to the twist $J$. Then the twisted algebra $\mathcal{O}(G)_J$ is the
Weyl algebra generated by $x,y$ with $yx-xy=1$, so it does not have
finite dimensional modules. Hence,
$\Fun_{\Rep(G)}(\mathcal{M}(G,1),\mathcal{N})=0$.

Note that
there is no $2-$cocycle $\psi$ with values in $\mathbb{G}_m$ (there
is one with values in $\mathbb{G}_a$, namely,
$\psi((x_1,x_2),(y_1,y_2))=x_1y_2-x_2y_1$, but to make it take
values in $\mathbb{G}_m$, one needs to take exponential, which is not algebraic).

2) The classification of fiber functors on $\Rep(G)$, where $G$ is a unipotent algebraic group, is given in \cite{EG3}. See also \cite{EG2} for the construction of fiber functors on $\Rep(G)$ for
other algebraic groups. However, the classification of fiber functors is not known for $SL_n$, $n\ge 4$ (it is known for $n\le 3$ \cite{Oh1}, \cite{Oh2}). 
\end{remark}

\subsection{Semisimple module categories of rank $1$.} Recall that the set of equivalence classes of semisimple module
categories over $\Rep(G)$ of rank $1$ is in bijection with the set
of equivalence classes of tensor
structures on the forgetful functor $\Rep(G)\to \Vect$. Therefore,
Theorem \ref{modg} implies that the
conjugacy class of any pair $(H,\psi)$ for which the category
$\Corep(\mathcal{O}(H)_{\psi})$ is semisimple of rank $1$ gives rise
to an equivalence class of a tensor structure on the forgetful functor
$\Rep(G)\to \Vect$. Clearly, for such pair $(H,\psi)$, $H$ must be a
\emph{finite} group subscheme of $G$ (as a simple coalgebra must be finite dimensional). This observation suggests the
following definition.

\begin{definition}\label{nondeg}
Let $H$ be a finite group scheme over $k$. We call a $2-$cocycle $\psi:H\times H\to \mathbb{G}_m$ (equivalently, a twist
$\psi$ for $\mathcal{O}(H)=k[H]^*$) \emph{nondegenerate} if the
category $\Corep(\mathcal{O}(H)_{\psi})$ of finite dimensional
comodules over $\mathcal{O}(H)_{\psi}$ is equivalent to $\Vect$
(i.e., if the coalgebra $\mathcal{O}(H)_{\psi}$ is simple).
\end{definition}

We thus have the following corollary.

\begin{corollary}\label{twists}
The conjugacy class of a pair $(H,\psi)$, where $H$ is a finite
closed group subscheme of $G$ and $\psi:H\times H\to \mathbb{G}_m$
is a nondegenerate $2-$cocycle, gives rise to an equivalence class
of a Hopf $2-$cocycle for $\mathcal{O}(G)$. \qed
\end{corollary}
 
\begin{remark}
Finite group schemes having a nondegenerate $2-$cocycle may be
called group schemes of \emph{central type} in analogy with finite
abstract groups.
\end{remark}

\section{group scheme theoretical categories}\label{gth}

In this section we extend the classes of rigid tensor categories $\Rep(G)$ and $\text{Coh}_f(G)$ to a larger class of {\em group scheme theoretical categories}, exactly in the same way as it is done for finite groups \cite{O}.

Let $G$ be an affine group scheme over $k$, and let $\omega\in H^3(G,\mathbb{G}_m)$ be a normalized $3-$cocycle. Equivalently, $\omega\in \mathcal{O}(G)^{\ot 3}$ is a
\emph{Drinfeld associator} for $\mathcal{O}(G)$, i.e., $\omega$ is an
invertible element satisfying the equations
$$(id\ot id\ot \Delta)(\omega)(\Delta\ot id\ot id)(\omega)=(1\ot \omega)(id\ot \Delta\ot id)(\omega)(\omega\ot 1),$$
$$(\varepsilon\ot id\ot id)(\omega)=(id\ot \varepsilon\ot id)(\omega)=(id\ot id\ot \varepsilon)(\omega)=1.$$

The proof of the following lemma is straightforward.

\begin{lemma}\label{omega}
The category $\text{Coh}_f(G)$ (resp., $\text{Coh}(G)$) with tensor product given by convolution of sheaves and associativity constraint given by the action of $\omega$ (viewed as an invertible element in $\mathcal{O}(G)^{\ot 3}$) is a rigid tensor category (resp., tensor category). \qed
\end{lemma}

Let us denote the rigid tensor category (resp., tensor category) from Lemma \ref{omega} by $\text{Coh}_f(G,\omega)$ (resp., $\text{Coh}(G,\omega)$).

Let $H$ be a closed group subscheme of $G$, and let $\psi\in C^2(H,\mathbb{G}_m)$ be a normalized $2-$cochain such that $d\psi=\omega_{|H}$. Let $\text{Coh}_f^{(H,\psi)}(G,\omega)$ be the category of $(H,\psi)-$equivariant coherent sheaves on $(G,\omega)$ \emph{with finite support in
$G/H$}; it is defined similarly to $\text{Coh}_f^{(H,\psi)}(G)$ (the case $\omega=1$) with the obvious adjustments. The proof of the following lemma is similar to the proof of Lemma \ref{simpleex}.

\begin{lemma}\label{momega}
The category $\text{Coh}_f^{(H,\psi)}(G,\omega)$ admits a structure of an indecomposable exact module category over $\text{Coh}_f(G,\omega)$ given by convolution of sheaves. \qed
\end{lemma}

The proof of the following classification result is similar to the proof of Theorem \ref{grsch1}.

\begin{theorem}\label{modomega}
There is a bijection between conjugacy classes of pairs $(H,\psi)$ and
equivalence classes of indecomposable exact module categories over
$\text{Coh}_f(G,\omega)$, assigning $(H,\psi)$ to $\text{Coh}_f^{(H,\psi)}(G,\omega)$. \qed
\end{theorem}

Let us denote by $\C(G,H,\omega,\psi)$ the dual category of $\text{Coh}_f(G,\omega)$ with respect to its indecomposable exact module category $\text{Coh}_f^{(H,\psi)}(G,\omega)$. We have that $\C(G,H,\omega,\psi)$ is
equivalent to the tensor category of $((H,\psi),(H,\psi))-$biequivariant sheaves on $(G,\omega)$, supported on finitely many
left $H-$cosets (equivalently, right $H-$cosets),
with tensor product given by convolution of sheaves.

\begin{definition}
A rigid tensor category which is tensor equivalent to some $\C(G,H,\omega,\psi)$ is called a group scheme theoretical category.
\end{definition}

\begin{example} 
1) Both $\Rep(G)$ and $\text{Coh}_f(G,\omega)$ are group scheme theoretical categories. 

2) ({\em The center}) 
The center $\mathcal{Z}(\text{Coh}_f(G))$ of $\text{Coh}_f(G)$ is a group scheme theoretical category. Indeed, it is tensor equivalent to $\C(G\times G,G,1,1)$, where $G$ is considered as a closed group subscheme of $G\times G$ via the diagonal morphism $\Delta:G\to G\times G$.

For finite groups $G$, $\text{Coh}_f(G)=\Rep(\Fun(G))$, so  $\mathcal{Z}(\Rep(\Fun(G)))$ is equivalent to $\Rep(D(\Fun(G)))$ (where 
$D(\Fun(G))=\mathbb{C}[G]\ltimes \Fun(G)$ is the Drinfeld double), and it is well known that $\Rep(\mathbb{C}[G]\ltimes \Fun(G))$ is equivalent to the category of $G-$equivariant coherent sheaves on $G$.

In the algebraic group case, the same is true. Let us give an instructive example: suppose that $G$ is a semisimple adjoint algebraic group over $\mathbb{C}$ (i.e., with trivial center).
Let $\mathfrak{g}^*=\text{Lie}(G)^*$ be the coadjoint representation, regarded as a commutative algebraic group (multiple of $\mathbb{G}_a$). A coherent sheaf with finite dimensional space of global sections must be supported
on a finite conjugacy class, i.e., at $1$. So we are talking about $G-$equivariant coherent sheaves on $G$ supported (scheme-theoretically) at $1$.
This is the same as sheaves on $\mathfrak{g}$ with the same property (by using the exponential map), i.e., $G-$equivariant $S\mathfrak{g}^*-$modules, i.e., $G-$equivariant algebraic representations of $\mathfrak{g}^*$, which is the same as representations of the semidirect product $G\ltimes \mathfrak{g}^*$. Therefore, the center $\mathcal{Z}(\text{Coh}_f(G))$ is braided equivalent to $\Rep(G\ltimes \mathfrak{g}^*)$ equipped with its natural (nonsymmetric) braided structure. \footnote{The $R-$matrix for this category is $\exp(\sum x_i\otimes x_i^*)$, where $x_i$ is a basis of $\mathfrak{g}$ and $x_i^*$ the dual basis of $\mathfrak{g}^*$. Note that applying $R$ in $X\otimes Y$,
where $X,Y\in \Rep(G\ltimes \mathfrak{g}^*)$, the exponential will turn into a finite sum (almost all terms of the Taylor series of the exponential will be zero)
because $\mathfrak{g}^*$ acts nilpotently on $Y$.} 
\end{example}

\begin{definition}
Let $\C:=\C(G,H,\omega,\psi)$, $\mathcal{M}(H,\psi):=\text{Coh}_f^{(H,\psi)}(G,\omega)$.
A {\em geometrical} module category $\mathcal{N}$ over
$\C$ is an exact module category $\mathcal{N}$ such that $\Fun_{\C}(\mathcal{M}(H,\psi),\mathcal{N}_i)\ne 0$ for any indecomposable direct summand module category $\mathcal{N}_i$ of $\mathcal{N}$.
\end{definition}

It is clear that geometrical module categories over
$\C(G,H,\omega,\psi)$ form a full $2-$subcategory $\Mod_{geom}(\C(G,H,\omega,\psi))$ of the $2-$category $\Mod(\C(G,H,\omega,\psi))$ (see Remark \ref{2cat}).
 
The following extends Theorem \ref{modg} (see also Remark \ref{modgtf} in the next section).
 
\begin{theorem}\label{modgth}
Let $\C:=\C(G,H,\omega,\psi)$, $\mathcal{M}(H,\psi):=\text{Coh}_f^{(H,\psi)}(G,\omega)$. Then the $2-$functors
$$\Mod(\text{Coh}_f(G,\omega))\to \Mod_{geom}(\C),\,\,
\mathcal{M}\mapsto\Fun_{\text{Coh}_f(G,\omega)}(\mathcal{M}(H,\psi),\mathcal{M}),$$ 
and
$$\Mod_{geom}(\C)\to \Mod(\text{Coh}_f(G)),\,\,
\mathcal{N}\mapsto\Fun_{\C}(\mathcal{M}(H,\psi),\mathcal{N}),$$
are $2-$equivalences which are mutually inverse. 
In particular, there is
a bijection between conjugacy classes of pairs $(H',\psi')$, where
$H'$ is a closed group subscheme of $G$ and $\psi'\in
C^2(H',\mathbb{G}_m)$ satisfies $d\psi'=\omega_{|H'}$, and equivalence classes of indecomposable geometrical module categories over $\C(G,H,\omega,\psi)$.\qed
\end{theorem}

\section{exact module categories over finite group schemes}\label{fin}

In this section $G$ will denote a {\em finite} group scheme over $k$ (see Section 2.1.1).

\subsection{Module categories.} Thanks to \cite[Theorem 3.31]{EO}, Theorem \ref{modg} can be strengthened in the
finite case to give a canonical bijection between exact module
categories over $\text{Coh}_f(G)=\text{Coh}(G)$ and $\Rep(G)$ (i.e., for finite group schemes, every exact module category over $\Rep(G)$ is geometrical). Namely, we have the following result.

\begin{theorem}\label{modgfin}
Let $G$ be a finite group scheme over $k$. The $2-$functors
$$\Mod(\text{Coh}(G))\to \Mod(\Rep(G)),\,\,
\mathcal{M}\mapsto\Fun_{\text{Coh}(G)}(\mathcal{M}(G,1),\mathcal{M}),$$ 
and
$$\Mod(\Rep(G))\to \Mod(\text{Coh}(G)),\,\,
\mathcal{N}\mapsto\Fun_{\Rep(G)}(\Rep(\{1\},1),\mathcal{N}),$$
are $2-$equivalences which are mutually inverse.  In particular, the
equivalence classes of indecomposable exact module categories over
$\Rep(G)=\Rep(k[G])$ are parameterized by the conjugacy classes of
pairs $(H,\psi)$, where $H$ is a closed group subscheme of $G$ and
$\psi:H\times H\to \mathbb{G}_m$ is a normalized $2-$cocycle. \qed
\end{theorem}

\begin{remark}\label{modgtf}
More generally, the equivalence classes of indecomposable exact
module categories over $\C(G,H,\omega,\psi)$ are
parameterized by the conjugacy classes of pairs $(H',\psi')$, where
$H'$ is a group subscheme of $G$ and $\psi'\in
C^2(H',\mathbb{G}_m)$ satisfies $d\psi'=\omega_{|H'}$.
\end{remark}

\subsection{Twists for $k[G]$.} By \cite{AEGN}, there is a bijection between nondegenerate
twists for $k[G]$ and nondegenerate twists for $\mathcal{O}(G)$. Hence, as a consequence of Theorem \ref{modgfin}, we
deduce the following strengthening of Corollary \ref{twists}.

\begin{corollary}\label{twistsfinite}
Let $G$ be a finite group scheme over $k$. The following four sets
are in canonical bijection one with the other:

1) The set of equivalence classes of tensor structures
on the forgetful functor on $\Rep(G)$.

2) The set of gauge equivalence classes of twists for $k[G]$.

3) The set of conjugacy classes of pairs $(H,\psi)$, where
$H$ is a closed group subscheme of $G$ and $\psi:H\times H\to
\mathbb{G}_m$ is a nondegenerate $2-$cocycle.

4) The set of conjugacy classes of pairs $(H,J)$, where $H$
is a closed group subscheme of $G$ and $J$ is a nondegenerate twist
for $k[H]$. \qed
\end{corollary}

\begin{remark}\label{finite}
Corollary \ref{twistsfinite} was proved for etale group schemes in
\cite{Mov}, \cite{EG} and \cite{AEGN}.
\end{remark}

\subsection{Minimal twists for $k[G]$.} Recall that a twist $J$ for $k[G]$ is called \emph{minimal} if the
triangular Hopf algebra $(k[G]^J,J_{21}^{-1}J)$ is minimal, i.e., if
the left (right) tensorands of $J_{21}^{-1}J$ span $k[G]$ \cite{R}. 

Using Deligne's theorem \cite{DE}, it is shown in \cite{EG, AEGN} that a twist for a finite abstract group is minimal if and only if it is nondegenerate. In this section we show that the same holds for any finite group scheme, without using Deligne's theorem. In order to achieve it, we shall need the following result about quotients of Tannakian categories, which is of interest by itself.\footnote{We are grateful to the referee for pointing to us that this result is a special case of \cite[Proposition 1]{B}. We include our proof for the sake of completeness, and convenience of the reader.}
\begin{proposition}\label{quot}
Let $G$ be a finite group scheme over $k$, let $\C$ be a symmetric rigid tensor category over $k$, and suppose there exists a surjective \footnote{By saying that $F$ is {\em surjective} we mean that any object $X\in \C$ is isomorphic to a subquotient of $F(V)$ for some $V\in \Rep(G)$.} symmetric tensor functor $F:\Rep(G)\to \C$. Then there exists a closed group subscheme $H$ of $G$ such that $\C\cong \Rep(H)$ as symmetric rigid tensor categories, and $Forget_G\cong Forget_H\circ F$.
\end{proposition}

\begin{proof}
Consider the image $F(\mathcal{O}(G))$ of the commutative unital algebra object $\mathcal{O}(G)$ in $\Rep(G)$; it is a commutative unital algebra object in $\C$. Let $I\in\C$ be a maximal ideal subobject of $F(\mathcal{O}(G))$, and set $R:=F(\mathcal{O}(G))/I$. Then $R$ is a commutative unital algebra object in $\C$. Let $\Mod_{\C}(R)$ be the category of modules in $\C$ over $R$. Clearly, $R$ is a simple object in $\Mod_{\C}(R)$, so $\Hom_{\C}(\mathbf{1},R)=\Hom_{R}(R,R)=k$. 

Observe that for any $X\in \C$, there exists a finite dimensional vector space $\overline{X}$ such that $X\ot R\cong \overline{X}\ot _k R$ as modules over $R$ (i.e., $X\ot R$ is free). Indeed, this follows since $\mathcal{O}(G)\in \Rep(G)$ has this property and $F$ is surjective. Therefore, since $\Hom_{\C}(\mathbf{1},R)=k$, it follows that $\Hom_R(R,X\otimes R)=\Hom_{\C}(\mathbf{1},X\otimes R)$ canonically by Frobenius reciprocity, which implies that there is a canonical isomorphism $\Hom_{\C}(\mathbf{1},X\ot R)\ot _k R\xrightarrow{\cong}X\otimes R$. Hence the functor $$L:\C\to \Vect,\,X\mapsto \Hom_{\C}(\mathbf{1},X\ot R),$$ together with the tensor structure given by
\begin{eqnarray*}
\lefteqn{L(X\ot Y)=\Hom_{\C}(\mathbf{1},(X\ot Y)\ot R)}\\& &\xrightarrow{\cong}\Hom_{\C}(\mathbf{1},X\ot (L(Y)\ot _k R))\xrightarrow{\cong}\Hom_{\C}(\mathbf{1},(L(X)\ot _k L(Y))\ot _k R))\\
& & \xrightarrow{\cong} L(X)\ot _k L(Y), 
\end{eqnarray*}
is a fiber (= exact tensor) functor on $\C$.
But then a standard argument (see e.g., \cite{DM}) yields that $\C$ is equivalent to $\Rep(A)$ for some finite dimensional Hopf algebra $A$ over $k$, as a rigid tensor category. Hence, there exists an injective homomorphism $A\xrightarrow{1-1} k[G]$ of Hopf algebras, and the result follows.      
\end{proof}

\begin{remark}
Proposition \ref{quot} holds for any affine group scheme over $k$ (i.e., not necessarily finite). Namely, quotients of Tannakian categories are Tannakian. The proof is essentially the same, except that $\mathcal{O}(G)$ and its image under (the extension of) $F$ are Ind objects, so certain adaptations are required (see \cite[Proposition 1]{B}). 
\end{remark}

We can now state and prove the main result of this section.

\begin{proposition}\label{minond}
Let $G$ be a finite group scheme over $k$, and let $J$ be a twist for $k[G]$. Then $J$ is minimal if and only if it is nondegenerate.
\end{proposition}

\begin{proof}
Suppose $J$ is minimal. By Corollary \ref{twistsfinite}, there exist a closed group
subscheme $\overline{H}$ of $H$ and a nondegenerate twist
$\overline{J}$ for $k[\overline{H}]$ such that the image of
$\overline{J}$ under the embedding
$k[\overline{H}]_{\overline{J}}\hookrightarrow k[H]_J$ is $J$. Since
$J$ is minimal and $\overline{H}\subseteq H$, it follows that
$\overline{H}=H$.

Conversely, suppose $J$ is nondegenerate.
Let $(A,J_{21}^{-1}J)$ be the minimal triangular Hopf subalgebra of
$(k[G]^{J},J_{21}^{-1}J)$. The restriction functor $\Rep(G)\to \Rep(A)$ is a surjective symmetric tensor functor. Thus by Proposition \ref{quot},
$\Rep(A)$ is equivalent to $\Rep(H)$, as a symmetric tensor category, for some closed group subscheme $H$ of $G$.
Now, it is a standard fact
(see e.g., \cite{G}) that such an equivalence functor gives rise to a twist $I\in
k[H]^{\ot 2}$ and an isomorphism of triangular Hopf algebras
$
(k[H]^I,I_{21}^{-1}I)\xrightarrow{\cong} (A,J_{21}^{-1}J)$.

We therefore get an injective homomorphism of triangular Hopf algebras
$(k[H]^{I},I_{21}^{-1}I)\xrightarrow{1-1} (k[G]^{J},J_{21}^{-1}J)$, which implies
that $JI^{-1}$ is a symmetric twist for $k[G]$. But by \cite[Theorem 3.2]{DM}, this implies that $JI^{-1}$ is gauge equivalent to $1\ot 1$. Therefore, the triangular Hopf algebras
$(k[G]^{JI^{-1}},I_{21}J_{21}^{-1}JI^{-1})$ and $(k[G],1\ot 1)$ are
isomorphic. In other words, $(k[G]^{I},I_{21}^{-1}I)$ and
$(k[G]^{J},J_{21}^{-1}J)$ are isomorphic as triangular Hopf algebras, i.e., the pairs $(G,J)$ and $(H,I)$ are conjugate. We thus conclude from Corollary \ref{twistsfinite} that $H=G$, and hence that $J$ is a minimal twist, as required.
\end{proof}

\subsection{The commutative case}\label{excom}

Let $A$ be a finite commutative group
scheme over $k$ and let $A^D$ be its Cartier dual (see Section \ref{gsch}). By definition,
$k[A]=\mathcal{O}(A^D)$ and $k[A^D]=\mathcal{O}(A)$. Therefore,
Corollary \ref{twistsfinite} implies the following.

\begin{proposition}\label{twistscommut}
There is a canonical isomorphism of abelian groups between the group
of gauge equivalence classes of twists for $k[A]$ and the group
$H^2(A^D,\mathbb{G}_m)$. \qed
\end{proposition}

\begin{corollary}\label{vanish1}
Suppose that either $A=A_{ec}$ or $A=A_{ce}$. Then the equivalence
classes of indecomposable exact module categories over $\Rep(A)$ are
in bijection with the conjugacy classes of closed group subschemes
of $A$. In particular, the trivial twist is the only twist for
$k[A]$, i.e., the forgetful functor on $\Rep(A)$ has only the trivial tensor structure.
\end{corollary}

\begin{proof}
By Proposition \ref{twistscommut}, it is sufficient to show that in
both cases $H^2(A,\mathbb{G}_m)=0$. Indeed, consider the group
homomorphism
$$H^2(A,\mathbb{G}_m)\to \Hom(A\times A,\mathbb{G}_m),\,\,\psi\mapsto
\psi_{21}^{-1}\psi;$$ it is well defined since for any two choices
$\psi_1,\psi_2$ the $2-$cocycle $\psi_1\psi_2^{-1}$ is symmetric,
and it is known that
$H_s^2(A,\mathbb{G}_m)=\Ext^1(A,\mathbb{G}_m)=0$ (see, e.g.,
\cite{Mum}). Clearly, its image is contained in the group of
skew-symmetric bilinear forms on $A$, i.e., is contained in
$\Hom(A,A^D)$. But since $\Hom(A_{ec},A_{ce})=0$, $\Hom(A,A^D)=0$,
so the above homomorphism is trivial. This means that
$H^2(A,\mathbb{G}_m)=H_s^2(A,\mathbb{G}_m)=0$, as claimed.
\end{proof}

In contrast, the cases $A=A_{ee}$ (see Remark \ref{finite}) and
$A=A_{cc}$ (see Example \ref{abelianres}) are more interesting, as
demonstrated also by the following proposition.

\begin{proposition}\label{nonvanish1}
Let $\psi\in H^2(A\times A^D,\mathbb{G}_m)$ be the class represented by the $2-$cocycle given by $\psi((a_1,f_1),(a_2,f_2))=\,<f_1,a_2>$,
where \linebreak $<\,,\,>$ denotes the canonical
pairing $<\,,\,>:A^D\times A\to \mathbb{G}_m$. Then $\psi$ is
nondegenerate, i.e., it corresponds to a nondegenerate twist for
$k[A\times A^D]$.
\end{proposition}

\begin{proof}
It is straightforward to verify that the corresponding
twist for $\mathcal{O}(A\times A^D)$ (which we shall also denote by $\psi$) is given by $\psi=\sum f_i\otimes a_i$, where $f_i$ and $a_i$
are dual bases of $\mathcal{O}(A)$ and $\mathcal{O}(A)^*$,
respectively. But $(\mathcal{O}(A\times A^D)_{\psi})^*$ is an
Heisenberg double, hence a simple algebra by [Mon, Corollary 9.4.3]
(see \cite{AEGN}), so $\psi$ is nondegenerate.
\end{proof}

\subsection{$p-$Lie algebras}\label{explie}

Assume that $k$ has characteristic $p>0$. In the case of $p-$Lie
algebras (see Section \ref{plie}), Theorem \ref{modgfin} and Corollary \ref{twistsfinite}
translate into the following result.

\begin{theorem}\label{modgfinlie}
Let $\mathfrak{g}$ be a finite dimensional $p-$Lie algebra over $k$.
The equivalence classes of indecomposable exact module categories
over $\Rep(\mathfrak{g})$ are parameterized by the conjugacy classes
of pairs $(\mathfrak{h},\psi)$, where $\mathfrak{h}$ is a $p-$Lie
subalgebra of $\mathfrak{g}$ and $\psi$ is a Hopf $2-$cocycle for
$u(\mathfrak{h})$. In particular, the gauge equivalence classes of
twists for $u(\mathfrak{g})$ are in bijection with conjugacy classes
of pairs $(\mathfrak{h},J)$, where $\mathfrak{h}$ is a $p-$Lie
subalgebra of $\mathfrak{g}$ and $J$ is a nondegenerate twist for
$u(\mathfrak{h})$. \qed
\end{theorem}

\begin{example}\label{abelianres} (Semisimple $p-$Lie algebras) Let
$\mathfrak{t}$ be a torus, i.e., the $p-$Lie algebra of a connected
diagonalizable group scheme. Then Corollary \ref{vanish1} tells us
that the forgetful functor on $\Rep(\mathfrak{t})$ has only the trivial tensor structure,
which is in contrast with the etale case.
\end{example}

\begin{example}\label{abelianres}
Let $\mathfrak{a}$ be the $2-$dimensional abelian $p-$Lie algebra
with basis $h,x$ such that $h^p=0$ and $x^p=0$ (it is the $p-$Lie
algebra of the group scheme $\alpha_{p}\times \alpha_{p}$). Then it
is straightforward to verify that
$$J:=\exp(h\otimes x)=\sum_{i=0}^{p-1}\frac{h^i\otimes x^i}{i!}$$ is
a nondegenerate twist for $u(\mathfrak{a})$. In fact, the algebra
$(u(\mathfrak{a})_J)^*$ is isomorphic to the truncated Weyl algebra
$k[x,y]/(xy-yx-1,x^p,y^p)$, which is known to be a simple algebra (\cite[p.73]{S}).
\end{example}

\begin{example}\label{2dim}
Let $\mathfrak{g}$ be the unique $2-$dimensional nonabelian $p-$Lie
algebra with basis $x,y$ such that $[x,y]=y$, $x^p=x$ and $y^p=0$
(it is the $p-$Lie algebra of the Frobenius kernel of the group
scheme $\mathbb{G}_{m}\ltimes \mathbb{G}_{a}$ of automorphisms of the affine line 
$\mathbb{A}^1$). It is straightforward to verify that the element
$$J:=\sum_{i=0}^{p-1}\frac{x(x-1)\cdots (x-i+1)\otimes y^i}{i!}$$ is
a nondegenerate twist for $u(\mathfrak{g})$, and that
$(u(\mathfrak{g})^J,J_{21}^{-1}J)$ is a noncommutative and noncocommutative minimal triangular Hopf algebra of dimension $p^2$.
\end{example}

\begin{example}\label{frobenius} (Frobenius $p-$Lie algebras)
Let $\mathfrak{g}$ be a finite dimensional \emph{Frobenius} $p-$Lie algebra over $k$. By definition, this means that there exists a linear functional $\xi\in \mathfrak{g}^*$ such that the bilinear form $(x,y)\mapsto \xi([x,y])$ on $\mathfrak{g}$ is nondegenerate. Let $u_{\xi}(\mathfrak{g})$ be the associated reduced universal enveloping algebra, i.e., $u_{\xi}(\mathfrak{g})$ is the quotient algebra of the universal enveloping algebra $U(\mathfrak{g})$ of $\mathfrak{g}$ by the ideal generated by elements $x^p-x^{[p]}-\xi(x)^p1$, $x\in \mathfrak{g}$.
By a well known result of Premet-Skryabin \cite{PS}, $u_{\xi}(\mathfrak{g})$ is a simple algebra. Therefore any Frobenius $p-$Lie algebra possesses a nondegenerate twist.

Since there are nonsolvable Frobenius $p-$Lie algebras, we see that
there exist finite group schemes of central type which are
\emph{not} solvable (unlike in the etale case \cite{EG}). For example, the
$6-$dimensional $p-$Lie subalgebra $\mathfrak {g}$ of $\mathfrak
{gl}_3$, consisting of the matrices with zero last row, is not
solvable but is Frobenius (e.g., let $\xi\in \mathfrak {g}^*$ be
defined on the standard basis $E_{ij}$ by
$\xi(E_{12})=\xi(E_{23})=1$ and
$\xi(E_{11})=\xi(E_{13})=\xi(E_{21})=\xi(E_{22})=0$).
\end{example}

\begin{example} (The Witt $p-$Lie algebra)
Let $\mathfrak{w}$ be the $p-$dimensional $p-$Lie algebra with basis
$x_i$, $i\in \mathbb{F}_p$, such that $[x_i,x_j]=(j-i)x_{i+j}$,
$x_0^p=x_0$ and $x_i^p=0$ for $i\ne 0$. Note that for any $i\ne 0$,
the elements $x:=i^{-1}x_0$, $y:=ix_i$ span a $2-$dimensional
nonabelian $p-$Lie subalgebra of $\mathfrak{w}$. We thus obtain twists $J(i)$ for $u(\mathfrak{w})$, $i\in \mathbb{F}_p^{\times}$, as
in Example \ref{2dim}, and hence $p^p-$dimensional noncommutative
and noncocommutative triangular Hopf algebras
$u(\mathfrak{w})^{J(i)}$. (See \cite{Gr} for similar results.)
\end{example}

\section{isocategorical finite group schemes}\label{isoc}

Following \cite{EG1}, we say that two finite group schemes $G_1$,
$G_2$ over $k$ are \emph{isocategorical} if $\text{Rep}(G_1)$ is
equivalent to $\text{Rep}(G_2)$ as a tensor category (without regard
for the symmetric structure). Then $G_1$,
$G_2$ are isocategorical if and only if the Hopf algebras $k[G_1]$, $k[G_2]$ are twist equivalent \cite{EG1}.

\subsection{The construction of isocategorical finite group schemes.}
The construction of all finite group
schemes isocategorical to a fixed finite group scheme given below extends the one given in \cite{EG1} for etale groups (see also \cite{Da1,Da2}).

Let $G$ be finite group scheme over $k$, let $A$ be a commutative normal closed group subscheme of $G$, let $A^D$ be the Cartier dual of $A$ (see Section 2.1.1), and set $K:=G/A$. Let
$R:A^D\times A^D\to \mathbb{G}_m$ be a $G-$equivariant nondegenerate
skew-symmetric (i.e., $R(a,a)=0$ for all $a\in A^D$) bilinear form on $A^D$. It is known that the image of the
group homomorphism
$$H^2(A^D,\mathbb{G}_m)\to \Hom(A^D\times
A^D,\mathbb{G}_m),\,\,\psi\mapsto \psi\psi_{21}^{-1},$$ is the group
of skew-symmetric bilinear forms on $A^D$. Therefore, the form $R$ defines a class in
$H^2(A^D,\mathbb{G}_m)^K$ represented by any $2-$cocycle $J\in
Z^2(A^D,\mathbb{G}_m)$ such that $R=JJ_{21}^{-1}$.

Let
\begin{equation*}\label{tau}
\tau:H^2(A^D, \mathbb{G}_m)^K\to H^2(K,A)
\end{equation*}
be the homomorphism defined as follows. For $c\in H^2(A^D,
\mathbb{G}_m)^K$, let $J$ be a $2-$cocycle representing $c$. Then
for any $g\in K$, the $2-$cocycle $J^gJ^{-1}$ is a coboundary.
Choose a cochain $z(g): A^D\to \mathbb{G}_m$ such that
$dz(g)=J^gJ^{-1}$, and let
\begin{equation*}\label{btild}
\tilde b(g,h):=z(gh)z(g)^{-1}(z(h)^g)^{-1}.
\end{equation*}
Then for any $g,h\in K$, the function $\tilde b(g,h):A^D\to
\mathbb{G}_m$ is a group homomorphism, i.e., $\tilde b(g,h)$ belongs
to the group $A$. Thus, $\tilde b$ can be regarded as a $2-$cocycle
of $K$ with coefficients in $A$. So $\tilde b$ represents a class
$b$ in $H^2(K,A)$, which depends only on $c$ and not on the choices
we made. So we define $\tau$ by $\tau (c)=b$.

Now, let $b:=\tau (R)$, and let $\tilde b$ be any cocycle
representing $b$. For any $\gamma\in G$, let $\bar \gamma$ be the
image of $\gamma$ in $K$. Introduce a new multiplication law $*$ on
the scheme $G$ by
\begin{equation*}\label{law}
\gamma_1*\gamma_2:=\tilde
b(\bar\gamma_1,\bar\gamma_2)\gamma_1\gamma_2.
\end{equation*}
It is easy to show that this multiplication law introduces a new
group scheme structure on $G$, which (up to an isomorphism) depends
only on $b$ and not on $\tilde b$. Let us call this finite group
scheme $G_b$.

\begin{theorem}\label{th2}
The following hold:

1) The finite group scheme $G_b$ is isocategorical to $G$.

2) Any finite group scheme isocategorical to $G$ is obtained in this way. 

In particular, two isocategorical finite group schemes are necessarily isomorphic as schemes (but not as groups \cite{EG1}).
\end{theorem}

\subsection{Sketch of the proof of Theorem \ref{th2}.}
Suppose that $G_1$ and $G_2$ are isocategorical, and fix a
twist $J$ for $k[G_1]$ such that $k[G_1]^J$ and $k[G_2]$ are
isomorphic as Hopf algebras (but {\em not} necessarily as triangular
Hopf algebras). Clearly, the Hopf algebra $k[G_1]^J$ is cocommutative. Set, $R^J:=J_{21}^{-1}J$.

Let $(k[G_1]^J)_{min}\subseteq k[G_1]^J$ be the minimal triangular
Hopf subalgebra of the triangular Hopf algebra $(k[G_1]^J,R^J)$
\cite{R}. Since $(k[G_1]^J)_{min}$ is isomorphic to its dual with
opposite coproduct (via $R^J$), $(k[G_1]^J)_{min}$ is cocommutative
and commutative. This implies that $(k[G_1]^{J})_{min}$ is
isomorphic to the group algebra $k[A]$ of a commutative group scheme
$A$. Therefore, there exists a twist $J'\in
(k[G_1]^{J})_{min}\otimes (k[G_1]^{J})_{min}$ such that
$R^J=R^{J'}$. But this implies (exactly as in the proof of
Proposition 3.4 in \cite{EG1}) that there exists a twist
$\widehat{J}$ for $k[G_1]$ such that $k[G_1]^J$ is isomorphic to
$k[G_1]^{\widehat{J}}$ as triangular Hopf algebras, and
$\widehat{J}\in (k[G_1]^{\widehat{J}})_{min}\otimes
(k[G_1]^{\widehat{J}})_{min}$.

Thus, we can assume, without loss of generality, that $J\in
(k[G_1]^{J})_{min}^{\otimes 2}$. This implies that
$(k[G_1]^{J})_{min}=k[A]$, where $A$ is a commutative closed group
subscheme of $G_1$, and $J\in k[A]\otimes k[A]$.

\begin{proposition} The closed group
subscheme $A$ is normal in $G_1$ (i.e., $k[A]$ is invariant under
the adjoint action $\Ad$ of $k[G_1]$ on itself), and the action of
the group scheme $K:=G_1/A$ on $A$ by conjugation preserves $R^J$.
\end{proposition}

\begin{proof} By cocommutativity of $k[G_1]^J$, $J^{-1}\Delta(g)J=
J_{21}^{-1}\Delta(g)J_{21}$ for all $g\in k[G_1]$, hence
$\Delta(g)R^J=R^J\Delta(g)$ (here we use that $k[A]$ is commutative,
so $R^J=JJ_{21}^{-1}$). But then, using the cocommutativity of
$k[G_1]^J$ again, we get that $R^J$ is invariant under the adjoint
action of $k[G_1]$, i.e., $\Ad (g)R^J=\varepsilon(g)R^J$ for all
$g\in k[G_1]$. Since the left (and right) tensorands of $R^J$ span
$k[A]$, the result follows.
\end{proof}

We can thus view $J$ not only as a twist for $k[A]$ but also as a
$2-$cocycle of $A^D$ with values in $\mathbb{G}_m$, according to
Proposition \ref{twistscommut}. For $g\in K$ let us write $J^g$ for
the action of $g$ on $J$. Since $R^J$ is invariant under $G_1$,
$J^gJ^{-1}=J_{21}^gJ_{21}^{-1}$, which implies that the $2-$cocycle
$J^gJ^{-1}:A^D\times A^D\to \mathbb{G}_m$ is symmetric. Hence there
exists a cochain $z(g): A^D\to \mathbb{G}_m$ (i.e., an invertible
element in $\mathcal{O}(A^D)=k[A]$) such that $J^gJ^{-1}=dz(g)$.

Identifying $k[G_1]^J$ with $k[G_2]$, we can consider the morphism of schemes
$\varphi:G_1\to G_2$, $\varphi(\gamma)=z(\gamma)^{-1}\gamma$ (where by $z(\gamma)$ we mean $z(\gamma A)$). Then
$\varphi$ is bijective (with inverse $\varphi^{-1}(\gamma)=z(\gamma)\gamma$).

Finally, it is obvious from the definition of $\varphi$ that
$$
\varphi(\gamma_1)\varphi(\gamma_2)=\tilde
b(\bar\gamma_1,\bar\gamma_2) \varphi(\gamma_1\gamma_2),
$$
where $\displaystyle{\tilde b(g,h)=z(gh)/z(g)z(h)^g\in
k[A]^\times}$. Furthermore, the morphism $\tilde b:=\tilde b(g,h):A^D\to
\mathbb{G}_m$ is a group homomorphism, i.e., $\tilde b\in A$, and it
is clear that $\tilde b$ is a $2-$cocycle of $K$ with coefficients
in $A$. Let $b$ be the cohomology class of $\tilde b$ in $H^2(K,A)$.
We have shown that
$$
\varphi(\gamma_1*\gamma_2)=\varphi(\gamma_1)\varphi(\gamma_2),
$$
i.e., that $\varphi$ is an isomorphism of group schemes $(G_1)_b\to G_2$. This
completes the proof of Part 2 of Theorem \ref{th2}, since by the
definition of $b$ we have $b=\tau (\overline{R^J})$.

Finally, Part 1 is essentially obvious from the above. Namely, if
$G$ is a finite group scheme, $A$ its commutative normal closed
group subscheme, $K:=G/A$ and $b:=\tau (\bar R)\in H^2(K,A)$, then
choose a twist $J\in k[A]^{\otimes 2}$ such that $R=J_{21}^{-1}J$
and get that $k[G]^J$ is isomorphic as a Hopf algebra to $k[G_b]$
(so the group schemes $G$ and $G_b$ are isocategorical). \qed

\end{document}